\documentclass[12pt]{article}
\usepackage{amsthm}
\usepackage{amsfonts}
\usepackage{amsmath}
\usepackage{amssymb}
\usepackage{hyperref}
\usepackage{epsfig}

\newtheorem{theorem}{Theorem}

\newtheorem{lemma}[theorem]{Lemma}

\title{Short proofs of coloring theorems on planar graphs}
\author{
Oleg V. Borodin\thanks{Sobolev Institute of Mathematics and Novosibirsk State University, Novosibirsk 630090, Russia E-mail: 
{\tt  brdnoleg@math.nsc.ru}. Research of this author is supported in part 
 by grants 12-01-00448 and 12-01-00631 of the Russian Foundation for Basic Research.}
\and
Alexandr V. Kostochka\thanks{University of Illinois at Urbana-Champaign, Urbana, IL 61801, USA and 
Sobolev Institute of Mathematics, Novosibirsk 630090, Russia. E-mail: {\tt kostochk@math.uiuc.edu}.
Research of this author is supported in part by NSF grant DMS-0965587.}
\and
Bernard Lidick\'y\thanks{University of Illinois at Urbana-Champaign, Urbana, IL 61801, USA, E-mail: {\tt lidicky@illinois.edu}}
\and
Matthew Yancey\thanks{University of Illinois at Urbana-Champaign, Urbana, IL 61801, USA, E-mail:  
{\tt yancey1@illinois.edu}.
Research of this author is partially supported by the Arnold O. Beckman Research Award of the University of Illinois at 
Urbana-Champaign and from National Science Foundation grant DMS 08-38434 ``EMSW21-MCTP: Research Experience for Graduate Students.''
}
}
\date{\today}
\begin{document}
\maketitle

\begin{abstract}
A recent lower bound on the number of edges 
in a $k$-critical $n$-vertex graph by Kostochka and Yancey 
yields a half-page proof of the celebrated
Gr\"otzsch Theorem  that every planar triangle-free graph is 3-colorable.
In this paper we use the same bound to give short proofs of other
known theorems on 3-coloring of planar graphs, among whose is
the Gr\" unbaum-Aksenov Theorem that every planar with at most three
triangles is $3$-colorable. We also prove the new result that every graph 
obtained from a triangle-free planar graph by adding a vertex of degree at most four
is $3$-colorable.
\end{abstract}

\section{Introduction}

Graphs considered in this paper are simple, i.e., without loops or parallel edges.
For a graph $G$, the set of its vertices is denoted by $V(G)$ and the set of  its
edges by $E(G)$.

An \emph{embedding} $\sigma$ of a graph $G=(V,E)$ in a surface $\Sigma$ is an injective 
mapping of 
$V$ to a point set $P$ in $\Sigma$ and $E$ to non-self-intersecting curves in $\Sigma$
such that 
(a) for all $v \in V$ and $e \in E$, $\sigma(v)$ is never an interior point of $\sigma(e)$,
and $\sigma(v)$ is an endpoint of $\sigma(e)$ if and only if $v$ is a vertex of $e$, and
(b) for all $e,h \in E$, $\sigma(h)$ and $\sigma(e)$ can intersect only in vertices of $P$.
A graph is \emph{planar} if it has an embedding in the plane. 
A graph with its embedding in the (projective) plane is a \emph{(projective) plane}  graph.
A cycle in a graph embedded in $\Sigma$ is \emph{contractible} if it splits $\Sigma$ into
two surfaces where one of them is homeomorphic to a disk.

 A \emph{(proper) coloring} $\varphi$ of a graph $G$ is a mapping from $V(G)$
to a set of colors $C$ such that $\varphi(u) \neq \varphi(v)$ whenever $uv \in E(G)$. 
A graph $G$ is \emph{$k$-colorable} if there exists a coloring of $G$ using at most $k$ colors. 
A graph $G$ is \emph{$k$-critical} if $G$ is not $(k-1)$-colorable but every proper subgraph
of $G$ is $(k-1)$-colorable. 
By definition, if a graph $G$ is not $(k-1)$-colorable then it contains
a $k$-critical subgraph. 

Dirac~\cite{dirac1957} asked to determine the minimum number of edges in a $k$-critical graph.
Ore conjectured~\cite{ore1967} that an upper bound obtained from Haj\'{o}s' construction is tight.
More details about Ore's conjecture can be found in \cite{gcp}[Problem 5.3] and
in~\cite{KY12}.
Recently, Kostochka and Yancey~\cite{KY12} confirmed Ore's conjecture for $k=4$ and showed
that the conjecture is tight in infinitely many cases for every  $k\geq 5$. In~\cite{KY12-only4} they gave a
2.5-page proof of the case $k=4$:

\begin{theorem}[\cite{KY12-only4}]\label{thm-main}
  If $G$ is a $4$-critical $n$-vertex graph then
\[ |E(G)| \geq \frac{5n - 2}{3} .\]
\end{theorem}

 Theorem~\ref{thm-main} yields a half-page proof~\cite{KY12-only4} of the 
celebrated Gr\"otzsch Theorem~\cite{grotzsch1959} that every planar triangle-free graph
is 3-colorable.
This paper presents short proofs of some other theorems on $3$-coloring of graphs close to planar.
Most of these results are generalizations of  Gr\"otzsch Theorem.

Examples of such generalizations are results of Aksenov~\cite{aksenov77} and Jensen and Thomassen~\cite{JT00}.
\begin{theorem}[\cite{aksenov77,JT00}]\label{thm-jte}
 Let $G$ be a triangle-free planar graph and $H$ be a graph such that $G=H-h$ for some edge $h$ of $H$.
 Then $H$ is 3-colorable.
\end{theorem}
\begin{theorem}[\cite{JT00}]\label{thm-jtv}
 Let $G$ be a triangle-free planar graph and $H$ be a graph such that $G=H-v$ for some vertex $v$ of degree 3.
 Then $H$ is 3-colorable.
\end{theorem}
We show an alternative proof of Theorem~\ref{thm-jte} and give a strengthening
of Theorem~\ref{thm-jtv}.
\begin{theorem}\label{thm-vertex}
 Let $G$ be a triangle-free planar graph and $H$ be a graph such that $G=H-v$ for some vertex $v$ of degree 4.
 Then $H$ is 3-colorable.
\end{theorem}

Theorems~\ref{thm-jte} and~\ref{thm-vertex}  yield a short proof of the following extension theorem that was used 
by Gr\"otzsch~\cite{grotzsch1959}. 
\begin{theorem}\label{thm-precol}
 Let $G$ be a triangle-free planar graph and $F$ be a face of $G$ of length at most 5. Then
 each 3-coloring of $F$ can be extended to a 3-coloring of $G$.
\end{theorem}

An alternative statement of Theorem~\ref{thm-jte} is that each coloring of two vertices of 
a triangle-free planar 
graph $G$ by two different colors can be extended to a 3-coloring of $G$.
Aksenov et al.~\cite{aksenovetal02} extended Theorem~\ref{thm-jte} by showing 
that each proper coloring of each induced subgraph on two vertices of $G$ extends to a 3-coloring of $G$.

\begin{theorem}[\cite{aksenovetal02}]\label{thm-precol2}
 Let $G$ be a triangle-free planar graph. Then each coloring of two non-adjacent vertices
 can be extended to a 3-coloring of $G$.
\end{theorem}
We show a short proof of Theorem~\ref{thm-precol2}.

Another possibility to strengthen  Gr\"otzsch's Theorem is to allow at most three triangles.
\begin{theorem}[\cite{aksenov,borodin1997,grunbaum1963}]\label{thm-aksenov}
 Let $G$ be a planar graph containing at most three triangles. Then $G$ is 3-colorable.
\end{theorem}
The original proof by Gr\"unbaum~\cite{grunbaum1963} was incorrect and a correct proof
was provided by Aksenov~\cite{aksenov}.
A simpler proof was given by Borodin~\cite{borodin1997}, 
but our proof is significantly simpler.

Youngs~\cite{youngs96} constructed triangle-free graphs in the projective plane that are not 3-colorable.
Thomassen~\cite{thomassen94} showed that  if $G$ is embedded in the projective plane
without contractible cycles of length at most 4 then $G$ is 3-colorable. We slightly strengthen
the result by allowing two contractible 4-cycles or one contractible 3-cycle.
\begin{theorem}\label{thm-projective}
Let $G$ be a graph embedded in the projective plane such that the embedding  has at most two
contractible cycles of length 4  or one contractible cycle of length three such that all other cycles of 
length at most 4 are non-contractible.  Then $G$ is 3-colorable.
\end{theorem}

It turned out that restricting the number of triangles is not necessary. 
Havel conjectured~\cite{havel1969} that there exists a constant $c$
such that if every pair of triangles in a planar graph $G$ is at distance at
least $c$ then $G$ is 3-colorable. The conjecture was proven true
by Dvo\v{r}\'ak, Kr\'al' and Thomas~\cite{dvorak09}.

Without restriction on triangles, Steinberg conjectured~\cite{steinberg93} 
that every planar graph without 4- and 5-cycles is 3-colorable. 
Erd\H{o}s suggested to relax the conjecture and asked for the smallest $k$
such that every planar graphs without cycles of length 4 to $k$ is 3-colorable.
The best known bound for $k$ is 7~\cite{BGRS2005}.
A cycle $C$ is {\em triangular} if it is adjacent to
a triangle other than $C$. In \cite{BGR10}, it is proved that every  planar graph without
triangular cycles of length from  $4$ to $7$ is  $3$-colorable, which implies all
results in  
\cite{BMR10,BGMR2009,BGRS2005,CRW07,CW08,LCW07,WC07b,WC07a,WMLW10,xu2006}.

We present the following result in the direction towards Steinberg's conjecture
with a Havel-type constraint on triangles. 
As a free bonus, the graph can be in the projective plane instead of the plane.

\begin{theorem}\label{thm-456}
Let $G$ be a 4-chromatic projective planar graph where every vertex is in at 
most one triangle. 
Then $G$ contains a cycle of length 4,5 or 6.
\end{theorem}

There are numerous other results on the Three Color Problem in the plane.
See a recent survey~\cite{B12} or a webpage maintained by Montassier
\url{http://janela.lirmm.fr/~montassier/index.php?n=Site.ThreeColorProblem}.

The next section contains proofs of the presented theorems and Section~\ref{sec-tight}
contains constructions showing that some of the theorems are best possible.

\section{Proofs}

\emph{Identification} of non-adjacent vertices $u$ and $v$ in a graph $G$
results in a graph $G'$ obtained from $G-\{u,v\}$ by adding a new vertex $x$
adjacent to every vertex that is adjacent to at least one of $u$ and $v$.

The following lemma is a well-known tool to reduce the number of 4-faces.
We show its proof for the completeness.
\begin{lemma}\label{lem-4face}
Let $G$ be a plane graph and $F=v_0v_1v_2v_3$ be a 4-face in $G$
such that $v_0v_2, v_1v_3 \not\in E(G)$.
Let $G_i$ be obtained from $G$ by identifying $v_i$ and $v_{i+2}$
where $i \in \{0,1\}$. If the number of triangles increases in both $G_0$
and $G_1$ then there exists a triangle $v_iv_{i+1}z$ for some $z \in V(G)$ and $i \in \{0,1,2,3\}$.
Moreover, $G$ contains vertices $x$ and $y$ not in $F$ such that
$v_{i+1}zxv_{i-1}$ and $v_{i}zyv_{i+2}$ are paths in $G$. Indices are modulo 4.  
See Figure~\ref{fig-4face}.
\end{lemma}

\begin{figure}
\begin{center}\includegraphics{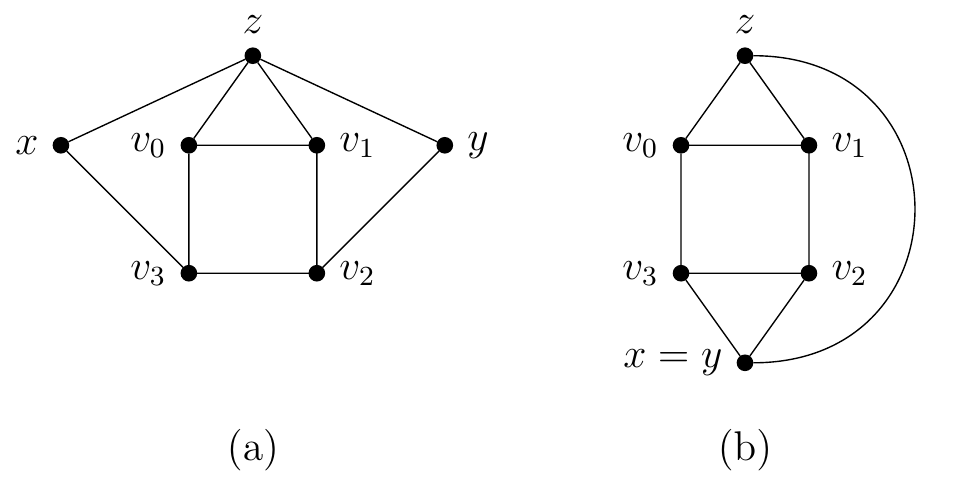}\end{center}
\caption{Triangle adjacent to a 4-face in Lemma~\ref{lem-4face}.\label{fig-4face}}
\end{figure}

\begin{proof}
Let $G, F, G_0$ and $G_1$ be as in the statement of the lemma. 
Since the number of triangles increases in $G_0$ there must be a path
$v_0zyv_2$ in $G$ where $z,y \not\in F$. 
Similarly, a new triangle in $G_1$ implies a path $v_1wxv_3$ in $G$ where $w,x \not\in F$.
By the planarity of $G$, $\{z,y\}$ and $\{w,x\}$ are not disjoint. 
Without loss of generality assume $z=w$. This results in triangle $v_0v_1z$ and
paths $v_{1}zxv_{3}$ and $v_{0}zyv_{2}$.
Note that $x$ and $y$ do not have to be distinct. See Figure~\ref{fig-4face}(b).
\end{proof}

\begin{proof}[Proof of Theorem~\ref{thm-jte}]
Let $H$ be a smallest counterexample and $G$ be a plane triangle-free
graph such that $G = H-h$ for some edge $h=uv$. 
Let $H$ have $n$ vertices and $e$ edges and $G$ have $f$ faces.
Note that $G$ has $n$ vertices and $e-1$ edges.
By the minimality of $H$, $H$ is 4-critical. So Theorem~\ref{thm-main} implies
$e \geq \frac{5n-2}{3}$.

CASE 1:
$G$ has at most one 4-face. Then $5f-1\leq 2(e-1)$ and hence $f \leq (2e-1)/5$.
By this and Euler's Formula $n-(e-1)+f = 2$ applied on $G$ we have
$ 5n - 3e + 1 \geq 5$, i.e., $e \leq  \frac{5n - 4}{3}$. This contradicts Theorem~\ref{thm-main}.

CASE 2: Every 4-face of $G$ contains both $u$ and $v$ and there
are at least two such 4-faces $F_x = ux_1vx_2$ and $F_y = uy_1vy_2$.
If there exists $z \in \{x_1,x_2\} \cap \{y_1,y_2\}$ then $z$ has degree two in $G$
which contradicts the 4-criticality of $G$.

Let $G'$ be obtained from $G$ by identification of $x_1$ and $x_2$ into a new vertex $x$.
If $G'$ is not triangle-free then there is a path $P=x_1q_1q_2x_2$ in $G$ where
$q_1,q_2 \not\in F_x$.
Since $P$ must cross $uy_1v$ and $uy_2v$, we may assume that
$y_1=q_1$ and $y_2=q_2$. However, $y_1y_2 \not\in E(G)$.
This contradicts the existence of $P$. Hence $G'$ is triangle-free.
Let $H' = G'+h$.
By the minimality of $H$, there exists a 3-coloring $\varphi$ of $H'$.
This contradicts that $H$ is not $4$-colorable since $\varphi$ can be extended
to $H$ by letting $\varphi(x_1) = \varphi(x_2) = \varphi(x)$.

CASE 3:
$G$ has a 4-face $F$ with vertices $v_0v_1v_2v_3$ in the cyclic order where $h$
is neither $v_0v_2$ nor $v_1v_3$. Since $G$ is triangle-free, neither $v_0v_2$ nor $v_1v_3$
are edges of $G$. 
Lemma~\ref{lem-4face} implies that either $v_0$ and $v_2$ or $v_1$ and $v_3$
can be identified without creating a triangle.  Without loss of generality
assume that $G'$, obtained by from $G$ identification of  $v_0$ and $v_2$ to a new
vertex $v$, is triangle-free.
Let $H' = G'+h$.
By the minimality of $H$, there is a 3-coloring $\varphi$ of $H'$. 
The 3-coloring $\varphi$ can be extended to $H$ by letting $\varphi(v_0) = \varphi(v_2) = \varphi(v)$
which contradicts the 4-criticality of $H$.
\end{proof}

\begin{proof}[Proof of Theorem~\ref{thm-vertex}]
Let $H$ be a smallest counterexample and $G$ be a plane triangle-free
graph such that $G = H-v$ for some vertex $v$ of degree 4. 
Let $H$ have $n$ vertices and $e$ edges and $G$ have $f$ faces.
Then $G$ has $n-1$ vertices and $e-4$ edges.
By minimality, $H$ is 4-critical. 
So Theorem~\ref{thm-main} implies $e \geq \frac{5n-2}{3}$.

CASE 1:
$G$ has no 4-faces. Then $5f \leq 2(e-4)$ and hence $f \leq 2(e-4)/5$.
By this and Euler's Formula $(n-1)-(e-4)+f = 2$ applied to $G$, we have
$ 5n - 3e -8 \geq  -5$, i.e., $e \leq  \frac{5n - 3}{3}$. This contradicts Theorem~\ref{thm-main}.

CASE 2:
$G$ has a 4-face $F$ with vertices $v_0v_1v_2v_3$ in the cyclic order.
Since $G$ is triangle-free, neither $v_0v_2$ nor $v_1v_3$
are edges of $G$ and Lemma~\ref{lem-4face} applies.
Without loss of generality assume that $G_0$ obtained from $G$
by identification of $v_0$ and $v_2$ is triangle-free.

By the minimality of $H$, 
the graph obtained from $H$ by identification of $v_0$ and $v_2$ satisfies the assumptions
of the theorem and hence has a 3-coloring.
Then $H$ also  has a 3-coloring, a contradiction.
\end{proof}

\begin{proof}[Proof of Theorem~\ref{thm-precol}]
Let the 3-coloring of $F$ be $\varphi$.

CASE 1: $F$ is a 4-face where $v_0v_1v_2v_3$ are its vertices in cyclic order.

CASE 1.1: $\varphi(v_0)=\varphi(v_2)$ and $\varphi(v_1)=\varphi(v_3)$.
Let $G'$ be obtained from $G$ by adding a vertex $v$ adjacent to
$v_0,v_1,v_2$ and $v_3$. Since $G'$ satisfies the assumptions of Theorem~\ref{thm-vertex},
there exists a 3-coloring $\varrho$ of $G'$. In any such 3-coloring,
$\varrho(v_0)=\varrho(v_2)$ and $\varrho(v_1)=\varrho(v_3)$. Hence by renaming
the colors in $\varrho$ we obtain an extension of $\varphi$ to a 3-coloring of $G$.

By symmetry, the other subcase is the following.

CASE 1.2: $\varphi(v_0)=\varphi(v_2)$ and 
$\varphi(v_1) \neq \varphi(v_3)$.
Let $G'$ be obtained from $G$ by adding the edge $v_1v_3$.
Since $G'$ satisfies the assumptions of Theorem~\ref{thm-jte},
there exists a 3-coloring $\varrho$ of $G'$. In any such 3-coloring,
$\varrho(v_1) \neq \varrho(v_3)$ and hence $\varrho(v_0)=\varrho(v_2)$. 
By renaming the colors in $\varrho$ we obtain an extension of $\varphi$ 
to a 3-coloring of $G$.

CASE 2: $F$ is a 5-face where $v_0v_1v_2v_3v_4$ are its vertices in cyclic order.
Observe that up to symmetry there is just one coloring of $F$. So
without loss of generality assume that $\varphi(v_0) = \varphi(v_2)$ and
$\varphi(v_1) = \varphi(v_3)$. 

Let $G'$ be obtained from $G$ by adding a vertex $v$ adjacent to
$v_0,v_1,v_2$ and $v_3$. Since $G'$ satisfies the assumptions of Theorem~\ref{thm-vertex},
there exists a 3-coloring $\varrho$ of $G'$. Note that in any such 3-coloring
$\varrho(v_0)=\varrho(v_2)$ and $\varrho(v_1)=\varrho(v_3)$. Hence by renaming
the colors in $\varrho$ we can extend  $\varphi$ to a 3-coloring of $G$.
\end{proof}

\begin{proof}[Proof of Theorem~\ref{thm-precol2}]
Let $G$ be a smallest counterexample and let $u,v \in V(G)$ be the two non-adjacent 
vertices colored by $\varphi$.
If $\varphi(u) \neq \varphi(v)$ then the result follows from Theorem~\ref{thm-jte}
by considering graph obtained from $G$ by adding the edge $uv$.
Hence assume that $\varphi(u) = \varphi(v)$.

CASE 1:
$G$ has at most two 4-faces. Let $H$ be a graph obtained from $G$ by identification
of $u$ and $v$. Any 3-coloring of $H$ yields a 3-coloring of $G$ where 
$u$ and $v$ are colored the same. By this and the minimality of $G$ we 
conclude that $H$ is 4-critical. Let $G$ have $e$ edges, $n+1$ vertices
and $f$ faces. 

Since $G$ is planar $5f-2\leq 2e$. By this and Euler's formula,
$$2e+2 + 5(n+1) -  5e \geq 10$$ 
and hence $e \leq (5n - 3)/3$, a contradiction to  Theorem~\ref{thm-main}.

CASE 2:
$G$ has at least three 4-faces. Let $F$ be a 4-face with vertices 
$v_0v_1v_2v_3$ in the cyclic order.
Since $G$ is triangle-free, neither $v_0v_2$ nor $v_1v_3$
are edges of $G$.  Hence Lemma~\ref{lem-4face} applies.

Without loss of generality let $G_0$ from Lemma~\ref{lem-4face} be triangle-free.
By the minimality of $G$, $G_0$ has a 3-coloring $\varphi$ where $\varphi(u) = \varphi(v)$ unless
$uv \in E(G_0)$. 
Since $uv \not\in E(G)$, without loss of generality $v_0=u$ and $v_2v \in E(G)$. 
Moreover, the same cannot happen to $G_1$ from Lemma~\ref{lem-4face}, hence $G_1$ 
contains a triangle. 
Thus $G$ contains a path $v_1q_1q_2v_3$ where $q_1,q_2 \not\in F$, and $G$ also contains a 
5-cycle $C = uv_1q_1q_2v_3$ (see Figure~\ref{fig-precol2}).
By Theorem~\ref{thm-precol}, $C$ is a 5-face. Hence $v$ is a 2-vertex incident with only one 4-face.

\begin{figure}
\begin{center}\includegraphics{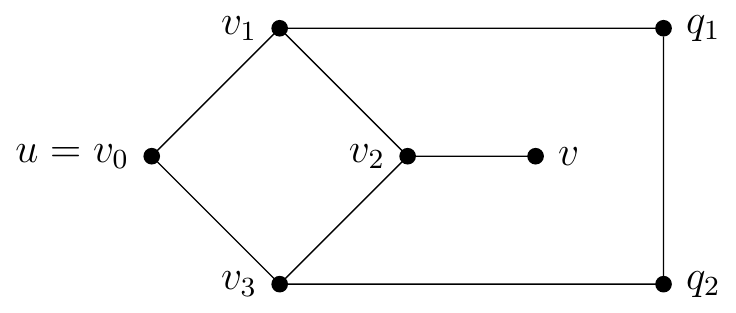}\end{center}
\caption{Configuration from Theorem~\ref{thm-precol2}.\label{fig-precol2}}
\end{figure}

By symmetric argument, $u$ is also a 2-vertex incident
with one 4-face and 5-face. However, $G$ has at least one more 4-face where
identification of vertices does not result in the edge $uv$, a contradiction
to the minimality of $G$.
\end{proof}

\begin{proof}[Proof of Theorem~\ref{thm-projective}]
Let $G$ be a minimal counterexample with $e$ edges, $n$ vertices and $f$ faces. 
By  minimality, $G$ is a 4-critical and has at most two 4-faces or one 3-face.
From embedding, $5f-2 \leq 2e$. By Euler's formula, $2e+2+ 5n-5e \geq  5$. 
Hence $e \leq (5n-3)/3$, a contradiction to  Theorem~\ref{thm-main}.
\end{proof}

Borodin used in his proof of Theorem~\ref{thm-aksenov} a technique called \emph{portionwise coloring}.
We avoid it and build the proof on the previous results arising
from Theorem~\ref{thm-main}. 

\begin{proof}[Proof of Theorem~\ref{thm-aksenov}]
Let $G$ be a smallest counterexample. By minimality, $G$ is 
4-critical and every triangle is a face. By Theorem~\ref{thm-precol} for
every separating 4-cycle and 5-cycle $C$, both the interior and exterior of $C$ contain  triangles.
 
CASE 1:
$G$ has no 4-faces. Then $5f-6 \leq 2e$ and by Euler's Formula 
$3e+6+5n-5e \geq 10$, i.e., $e \leq  \frac{5n - 4}{3}$. 
This contradicts Theorem~\ref{thm-main}.

CASE 2:
$G$ has a 4-face $F=v_0v_1v_2v_3$ such that $v_0v_2 \in E(G)$.
By the minimality, $v_0v_1v_2$ and  $v_0v_3v_2$ are both 3-faces
and hence $G$ has 4 vertices, 5 edges and it is 3-colorable.

CASE 3:
For every 4-face $F=v_0v_1v_2v_3$, neither $v_0v_2$ nor $v_1v_3$ are edges of $G$.
By Lemma~\ref{lem-4face}, there exist paths $v_0zyv_2$ and $v_1zxv_3$.

CASE 3.1:
$G$ contains a 3-prism with one of its 4-cycles
being a 4-face. We may assume that this face is our $F$ and $x=y$, see Figure~\ref{fig-4face}(b).
Theorem~\ref{thm-precol} implies that one of $zv_0v_3x$, $zv_1v_2x$
is a 4-face. Without loss of generality assume that $zv_1v_2x$ is a 4-face.
Let $G_0$ be obtained from $G$ by identification of $v_0$ and $v_2$ to a new vertex $v$.
Since $G_0$ is not 3-colorable, it contains a 4-critical subgraph $G'_0$.
Note that $G'_0$ contains triangle $xvz$ that is not in $G$ but $v_0v_1z$
is not in $G'_0$ since $d(v_1) = 2$ in $G_0$. 
By the minimality of $G$, there exists another triangle $T$ that is in 
$G_0$ but not in $G$. By planarity, $x\in T$. Hence there is a vertex $w_1 \neq v_3$
such that $v_0$ and $x$ are neighbors of $w_1$.

By considering identification of $v_1$ and $v_3$ and by symmetry,
we may assume that there is a vertex $w_2 \neq v_0$ such that $v_3$ and $z$ are neighbors of $w_2$.
By planarity we conclude that $w_1 = w_2$. This contradicts the fact that $G$ has
at most three triangles. Therefore $G$ is 3-prism-free.

CASE 3.2: $G$ contains no 3-prism with one of its 4-cycles
being a 4-face. Then
 $x \neq y$, see Figure~\ref{fig-4face}(a).
If $v_0x \in E(G)$ then $G - v_0$ is triangle-free and Theorem~\ref{thm-jtv}
gives a 3-coloring of $G$, a contradiction. Similarly, $v_1y \not\in E(G)$.

Suppose that $zv_0v_3x$ if a 4-face. Let $G'$ be obtained from $G-v_0$
by adding edge $xv_1$. If the number of triangles in $G'$ is at most three,
then $G'$ has a 3-coloring $\varphi$ by the minimality of $G$. Let $\varrho$
be a 3-coloring of $G$ such that $\varrho(v) = \varphi(v)$ if $v \in V(G')$
and $\varrho(v_0) = \varphi(x)$. Since the neighbors of $v_0$ in $G$ are neighbors
of $x$ in $G'$, $\varrho$ is a 3-coloring, a contradiction.
Therefore $G'$ has at least four triangles and hence $G$ contains a vertex
$t \neq z$ adjacent to $v_1$ and $x$. Since $v_1y \not\in E(G)$, the only
possibility is $t=v_2$. Having edge $xv_2$ results in a 3-prism being a
subgraph of $G$ which is already excluded. 
Hence $zv_0v_3x$ is not a face and by symmetry $zv_1v_2y$ is 
not a face either.

Since neither $zv_0v_3x$ nor  $zv_1v_2y$ is a face, each of them
contains a triangle in its interior. Since we know the location of all
three triangles, Theorem~\ref{thm-precol}
implies that $zyv_2v_3x$ is a 5-face. It also implies that 
the common neighbors of $z$ and $v_3$ are exactly $v_0$ and $x$,
and the common neighbors of $z$ and $v_2$ are exactly $v_1$ and $y$.
Without loss of generality, let $zyv_2v_3x$ be the outer face of $G$.

Let $H_1$ be obtained from the 4-cycle $zv_0v_3x$ and its interior
by adding edge $zv_3$. The edge $zv_3$ is in only two triangles,
and there is only one triangle in the interior of the 4-cycle. Hence by the minimality
of $G$, there exists a 3-coloring $\varphi_1$ of $H_1$.

Let $H_2$ be obtained from the 4-cycle $zv_1v_2y$ and its interior 
by adding edge $zv_2$. By the same argument as for $H_1$, there
is a 3-coloring of $\varphi_2$ of $H_2$.

Rename the colors in $\varphi_2$ so that
$\varphi_1(z) = \varphi_2(z)$, $\varphi_1(v_0)=\varphi_2(v_2)$ and  
$\varphi_1(v_3)=\varphi_2(v_1)$. Then $\varphi_1\cup \varphi_2$
is a  3-coloring of $G$, a contradiction.
\end{proof}

\begin{proof}[Proof of Theorem~\ref{thm-456}]
Let $G$ be a 4-chromatic projective plane graph where every vertex is in at most one triangle and
let $G$ be 4-,5- and 6-cycle free. Then $G$ contains a 4-critical subgraph $G'$.
Let $G'$ have $e$ edges, $n$ vertices and $f$ faces. 
Since $G'$ is also 4-,5- and 6-cycle-free and every vertex is in at most one triangle,
we get $ f \leq \frac{n}{3} + \frac{2e-n}{7}$.
By Euler’s formula, $7n +6e-3n + 21n - 21e \geq 21$. 
Hence $e \leq 5n/3 - 21/15$, a contradiction to  Theorem~\ref{thm-main}.
\end{proof}


\section{Tightness}\label{sec-tight}
This section shows examples where Theorems~\ref{thm-jte},\ref{thm-vertex},\ref{thm-precol},\ref{thm-precol2},\ref{thm-aksenov}, and \ref{thm-projective} are tight.

Theorem~\ref{thm-jte} is best possible because there exists an infinite family~\cite{thomwalls} 
of $4$-critical graphs that become triangle-free and planar after removal of just two edges.
See Figure~\ref{fig-thomas-wals}.
Moreover, the same family shows  also the tightness
of Theorem~\ref{thm-aksenov}, since the construction has exactly four triangles.

\begin{figure}
\begin{center}
\includegraphics{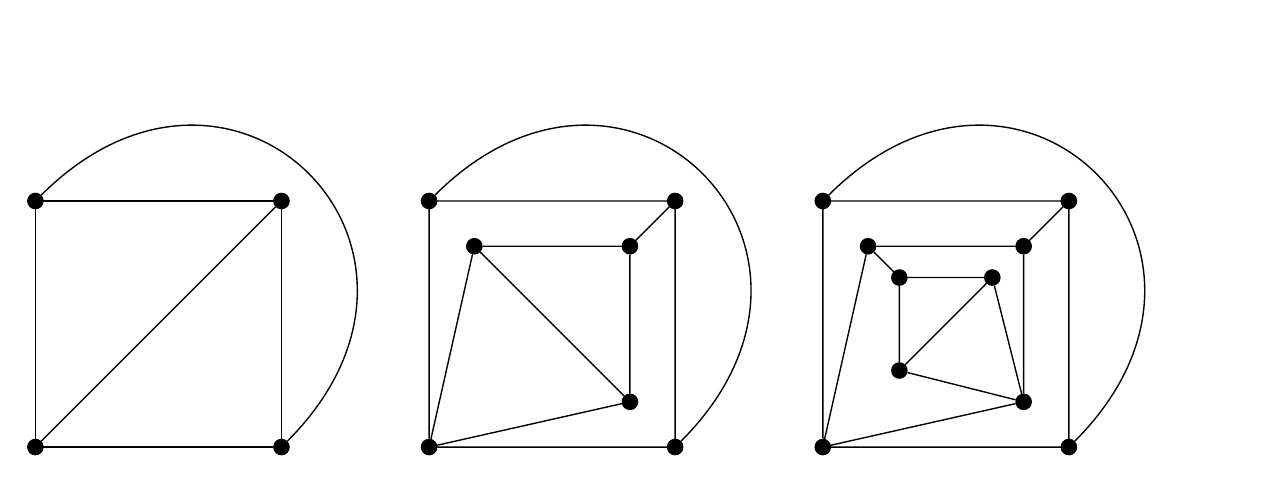}
\end{center}
\caption{First three 4-critical graphs from the family described by Thomas and Walls~\cite{thomwalls}.}
\label{fig-thomas-wals}
\end{figure}

Aksenov~\cite{aksenov} showed that every plane graph with
one 6-face $F$
and all other faces being 4-faces has no
3-coloring in which the colors of vertices of $F$ form the sequence $(1,2,3,1,2,3)$.
This implies that Theorem~\ref{thm-precol} is best possible.
It also implies that Theorems~\ref{thm-vertex} and \ref{thm-precol2} are best possible.
See Figure~\ref{fig-precol2-best} for constructions where coloring of three vertices
or an extra vertex of degree 5 force a coloring $(1,2,3,1,2,3)$ of a 6-cycle.

\begin{figure}
\begin{center}
\includegraphics{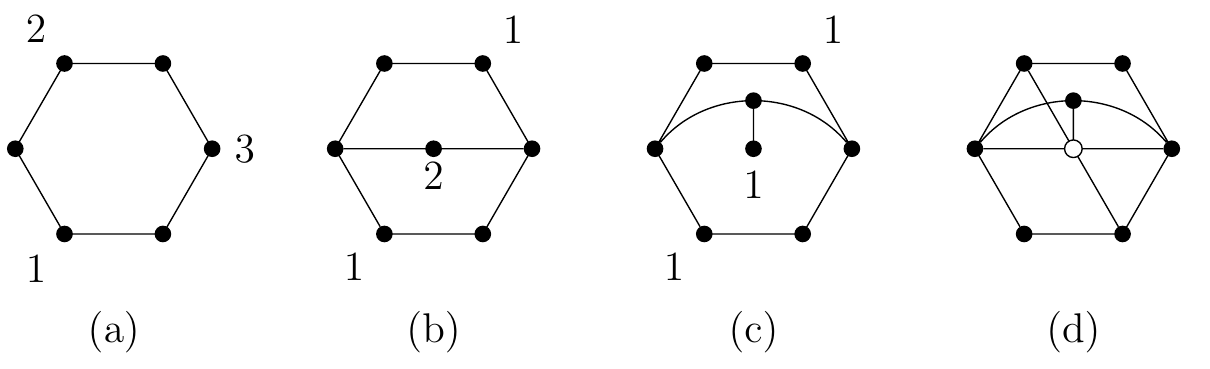}
\end{center}
\caption{Coloring of three vertices by colors 1, 2 and 3 in (a), (b) and (c) or an extra vertex of degree $5$
 in (d) forces a coloring of the 6-cycle by a sequence $(1,2,3,1,2,3)$ in cyclic order.}
\label{fig-precol2-best}
\end{figure}

Theorem~\ref{thm-projective} is best possible because there exist embeddings
of $K_4$ in the projective plane with three 4-faces or with two 3-faces and one 6-face.

\bibliographystyle{plain}

\end{document}